\documentclass{amsart}
\usepackage{amsfonts}
\usepackage{amsmath,amssymb}	
\usepackage{amsthm}
\usepackage{amscd}
\usepackage{graphics}
\usepackage{graphicx}

{
\theoremstyle{definition}
\newtheorem{Def}{Definition}
\newtheorem{Ex}{Example}
\newtheorem{Rem}{Remark}

\newtheorem{Conj}{Conjecture}
}

\newtheorem{Cor}{Corollary}
\newtheorem{Prop}{Proposition}
\newtheorem{Thm}{Theorem}

\newtheorem{Fact}{Fact}

\begin{document}
\title[obtaining special generic maps into the $3$-dimensional space]{A new explicit way of obtaining special generic maps into the $3$-dimensional Euclidean space}

\author{Naoki Kitazawa}
\address{19-9-606 Takayama, Tsuzuki-ku, Yokohama, Kanagawa 224-0065, JAPAN}
\email{naokikitazawa.formath@gmail.com}
\maketitle
\begin{abstract}
 A {\it special generic} map is a smooth map regarded as a natural generalization of Morse functions with just $2$ singular points on homotopy spheres.
Canonical projections of unit spheres are simplest examples of such maps and manifolds admitting special generic maps into the plane are completely determined by
 Saeki in 1993 and ones admitting such maps into general Euclidean spaces are determined under appropriate conditions.
 
Moreover, if the
 difference of dimensions of source and target manifolds are not so large, then the diffeomorphism types
 of source manifolds are often limited; for example, homotopy spheres except standard spheres do not admit special generic 
maps into Euclidean spaces whose dimensions are not so low. As another example, $4$-dimensional simply connected
 manifolds admitting special generic maps are only manifolds represented as the connected sum of the total
 spaces of $2$-dimensional sphere bundles over the $2$-dimensional sphere (or the $4$-dimensional
 standard sphere) and there are many manifolds homeomorphic and not diffeomorphic to
 these manifolds. These explicit facts make special generic maps attractive objects
 in the theory of Morse functions and higher dimensional
 versions and application to algebraic and differentiable topology of manifolds, which is an important study in both singuarity theory of maps and algebraic
 and differential topology of manifolds.  

In this paper, we demonstrate a way of construction of special generic maps into the $3$-dimensional Euclidean space.
For this, first we prepare maps onto $2$-dimensional polyhedra regarded as simplicial maps
 naturally called {\it pseudo quotient maps}, which are generalizations of the quotient maps to the spaces of all the
 connected components of inverse images, so-called {\it Reeb spaces} of original smooth maps, being fundamental and 
important tools in the studies. More precisely, we prepare specific pseudo quotient maps, locally construct maps onto $3$-dimensional manifolds, glue them and realize the
 map as a special generic map into the $3$-dimensional space. Most of the
 procedure for the construction is based on technique the author previously noticed and used in simpler situations.

The success of the construction explicitly shows that a class of maps which seems to cover a larger class of source manifolds may not be not so large and that
 the diffeomorphism types of source manifolds may be restricted as strongly as in the case of special generic maps. We
 also explain differential topological facts and problems related to this.

\end{abstract}

\section{Introduction.}
\label{sec:1}
\subsection{Backgrounds and fundamental tools.}
Morse functions and higher dimensional
 versions and application to algebraic and differentiable topology of manifolds is an important study in both singularity theory of maps and algebraic
 and differential topology of manifolds.  

A {\it fold map} is a smooth map such that each singular point is $p$ is
    of the form $(x_1,\cdots,x_m) \mapsto (x_1,\cdots,x_{n-1},\sum_{k=n}^{m-i(p)}{x_k}^2-\sum_{k=m-i(p)+1}^{m}{x_k}^2)$ for some
     integers $m,n,i(p)$. $i(p)$ is taken as a non-negative integer not larger than $\frac{m-n+1}{2}$ uniquely and we call $i(p)$ the {\it index} of $p$. The set of all the singular points of an index is a smooth submanifold of dimension $n-1$.
Morse functions are regarded as fold maps.
A fold map is {\it stable} or the $C^{\infty}$ equivalence classes of the maps are invariant under slight perturbations in the $C^{\infty}$ Whitney topology, if and only if
 the restriction to the set of all the
 singular points (of a fixed index), which are codimension $1$ smooth immersions, are transversal. Note that for example, stable Morse functions or Morse functions such that
 at distinct singular points, the values are distinct, exist densely. For Morse functions, fold maps and stable maps etc., see \cite{golubitskyguillemin} for example. For algebraic and
 differential topological properties of fold maps, see \cite{saeki} for example. 

{\it Special generic maps} are fold maps indices of whose singular points are $0$. They are regarded as simplest generalizations of Morse functions with just $2$ singular points on homotopy spheres. 
Canonical projections of unit spheres are simplest examples of such maps and manifolds admitting special generic maps into the plane are completely determined by
 Saeki in 1993 \cite{saeki2}, ones admitting such maps into general Euclidean spaces are determined under appropriate conditions and more precise facts have been shown. For example, if the difference of dimensions of source and target manifolds are not so large, then the diffeomorphism types
 of source manifolds are often limited; homotopy spheres except standard spheres do not admit special generic 
maps into Euclidean spaces whose dimensions are not so low. As another example, $4$-dimensional simply connected
 and closed manifolds admitting special generic maps into ${\mathbb{R}}^3$ are only manifolds represented as the connected sum of the total
 spaces of smooth $S^2$-bundles over the $2$-dimensional sphere (or the $4$-dimensional
 standard sphere) and there are many manifolds homeomorphic and not diffeomorphic to
 these manifolds; in \cite{saekisakuma} and \cite{saekisakuma2}, there are several $4$-dimensional closed manifolds admitting fold maps into ${\mathbb{R}}^3$ and
 admitting no special generic maps into ${\mathbb{R}}^3$ homeomorphic to ones admitting special generic maps into ${\mathbb{R}}^3$. 

These facts make special generic maps attractive objects from the viewpoint of differential topology of smooth manifolds. 

\subsection{Contents of this paper.}
 We define a {\it normal spherical} fold map. 
\begin{Def}
\label{def::1}
A stable fold map from a closed manifold of dimension $m$ into ${\mathbb{R}}^n$ satisfying $m \geq n$ is said to be {\it normal spherical} ({\it standard-spherical}) if the inverse image
 of each regular value is a disjoint union of (resp. standard) spheres (or points) and 
the connected component containing a singular point of the inverse image of a
 small interval intersecting with the singular value set at once in its interior is either of the following.
\begin{enumerate}
\item The ($m-n+1$)-dimensional standard closed disc.
\item A manifold PL homeomorphic to an ($m-n+1$)-dimensional compact manifold obtained by removing the interior of three disjoint ($m-n+1$)-dimensional smoothly embedded closed discs from the ($m-n+1$)-dimensional standard sphere.
\end{enumerate}  
\end{Def}
As a fundamental fact, this class includes special generic maps. 
Such fold maps are studied by Saeki and Suzuoka for example and for example, they can be obtained by projections of special generic maps by fundamental discussions
 of Saeki and Suzuoka \cite{saekisuzuoka}. Conversely, whether manifolds admitting
 such maps admit special generic maps into one dimensional higher Euclidean spaces is a natural and interesting problem, explicitly stated in the paper. Several answers
 have been given in \cite{kitazawa4} and \cite{kitazawa5} in the case of Morse functions.
In this paper, we consider such maps and extended ones and by applying technique based on ideas useful for obtaining the shown theorems for several
 cases of Morse functions, we obtain special generic maps. As additional remarks, we
 explain about diffeomorphsim types of source manifolds admitting such maps: for example, we will present the discovered fact that classes of source manifolds are not so large as conjectured and related topics on differential topology of manifolds. 

This paper is organized as follows. In the next section, we review Reeb spaces and we define {\it pseudo special generic maps} and as general
 PL maps, {\it pseudo quotient maps}, which were first introduced in \cite{kobayashisaeki}. Pseudo special generic maps have been introduced also in \cite{kitazawa3} and \cite{kitazawa5}. They are fundamental objects in the paper. 
After introducing these maps, we perform
 construction of special generic maps into ${\mathbb{R}}^3$; this is a main result. Last, we explain remarks on differential topological meanings of
 the obtained result and related facts and problems. 

Throughout the paper, all the manifolds and maps between them  and bundles over manifolds with fibers being manifolds etc. are smooth and of class $C^{\infty}$ unless otherwise stated. We also
 note that maps between polyhedra are PL unless otherwise stated in the paper.

Last, we call the set of all singular points of a smooth map the {\it singular set} of the map and we call
 the image of the singular set the {\it singular value set}. We call the set of all regular values the {\it regular value set}.
% Then we review a result and the essential technique for the proof and throughout 
% the paper. Last, we obtain a desired result.

\section{Reeb spaces, triangulable maps and two classes of maps on manifolds.}
\subsection{Reeb spaces and triangulable maps.}
We review Reeb spaces. For precise facts, see also \cite{reeb} and introduced papers. 
\begin{Def}
\label{def:2}
Let $X$ and $Y$ be topological spaces. For points $p_1, p_2 \in X$ and for a continuous map $c:X \rightarrow Y$, we define as $p_1 {\sim}_c p_2$ if and only if $p_1$ and $p_2$ are in
 the same connected component of $c^{-1}(p)$ for some $p \in Y$. The relation is an equivalence relation and we call the quotient space $W_c:=X/{\sim}_c$ the {\it Reeb space} of $c$.
\end{Def}
\begin{Ex}
\label{ex:2}
\begin{enumerate}
\item
\label{ex:2.1}
 For a Morse function, the Reeb space is a graph.
\item
\label{ex:2.2} 
 For a special generic map, the Reeb space is regarded as an immersed compact manifold with a non-empty boundary whose dimension is same as that of the target
 manifold. The Reeb space of a special generic map is known to be contractible
 if and only if the source manifold is a homotopy sphere. 

It is also a fundamental and important property that the quotient map from the source manifold onto the
 Reeb space induces isomorphisms on homology groups whose degrees are not larger than the difference of the dimension of the source manifold and that of the target one. See \cite{saeki2} for example. Moreover, we mention remarks on special generic maps in some parts of the present paper. 
\item
\label{ex:2.3}
 The Reeb spaces of stable (fold) maps are polyhedra whose dimensions and those of the target manifolds coincide. For
 a normal spherical Morse function, it is also a fundamental and important property that the quotient map from the source manifold onto the
 Reeb space induces isomorphisms on homology groups whose degrees are smaller than the difference of the dimension of the source manifold and that of
 the target one. See \cite{saekisuzuoka} and see also \cite{kitazawa} and \cite{kitazawa2} for example.
\item
\label{ex:2.4}
 For a (proper) stable map, the Reeb space is a polyhedron (\cite{shiota}).
\end{enumerate}
\end{Ex}

We mention {\it triangulable} maps.

\begin{Def}
\label{def:2}
Let $X$ and $Y$ be polyhedra. A continuous map $c:X \rightarrow Y$ is said to
 be {\it triangulable} if there exists a pair of triangulations of $X$ and $Y$ and homeomorphisms $({\phi}_X,{\phi}_Y)$ onto $X$ and $Y$ respectively such that the composition ${{\phi}_Y}^{-1} \circ c \circ {{\phi}_X}$ is a simplicial map with respect to the given triangulations. We also say that $c$ is triangulable with respect to $({\phi}_X,{\phi}_Y)$
\end{Def}

\begin{Fact}[\cite{shiota}]
\label{fact:1}
{\rm (}Proper{\rm )} stable maps are always triangulable {\rm (}with respect to pairs of homeomorphisms giving the canonical triangulations of the smooth manifolds{\rm )}.
\end{Fact}

\begin{Fact}[\cite{hiratukasaeki} and \cite{hiratukasaeki2}]
\label{fact:2}
For a triangulable map $c:X \rightarrow Y$ with respect to $(\phi_X,\phi_Y)$, the Reeb space $W_c$ is a polyhedron given by a homeomorphism ${\phi}_c$ from a polyhedron  and two maps $q_c:X \rightarrow W_c$
   and $\bar{c}:W_c \rightarrow Y$ are triangulable maps with respect to the corresponding pairs of the homeomorphisms.
\end{Fact}

\subsection{Two classes of PL or smooth maps.}
The following is introduced by the author \cite{kitazawa5}. 
\begin{Def}
\label{def:3}
Let $m>n$ be positive integers.
A smooth map $f_p$ from a closed manifold $M$ of dimension $m$ onto a compact manifold $W_p$ of dimension $n$ satisfying $\partial W_p \neq \emptyset$ is
 said to be {\it pseudo special generic} if the following hold.
\begin{enumerate}
\item $f_p {\mid}_{{f_p}^{-1}(W_p-\partial W_p)}:{f_p}^{-1}(W_p-\partial W_p) \rightarrow W_p-\partial W_p$ gives a smooth $S^{m-n}$-bundle.
\item For a small collar neighborhood $N(\partial W_p)$ of $\partial W_p$ in $W_p$, regarded as a trivial bundle $\partial W_p \times [0,1]$ where $\partial W_p \times \{0\}$
 corresponds to the boundary $\partial W_p$, for each point $(p,0) \in \partial W_p \times \{0\}$ and a small open neighborhood $U_p$,
 $f {\mid}_{{f_p}^{-1}(U_p \times [0,1])}:f^{-1}(U_p \times [0,1]) \rightarrow U_p \times [-1,1] $ has the same local form as a singular point of index $0$ of a fold map
 from an $m$-dimensional manifold into ${\mathbb{R}}^n$.
 From fundamental discussion of \cite{saeki} for example, ${f_p}^{-1}(N(\partial W_p))$ is a linear $D^{m-n+1}$-bundle over $\partial W_p$ given by the composition $f_p$ and the
 canonical projection and the bundle is seen as a normal bundle of the submanifold $f^{-1}(\partial W_p) \subset M$.
\end{enumerate}
\end{Def}
All the quotient maps onto the Reeb spaces defined from special generic maps into Euclidean spaces are regarded as pseudo special generic.
We also note that if the manifold $W_p$ can be immersed into the Eulidean space whose dimension is same as that of $W_p$, then by composing the immersion to the map, we
 have a special generic map. Special generic maps into Euclidean spaces are obtained like this, which is based on a fundamental discussion of \cite{saeki2}. 

We introduce {\it pseudo spherical fold maps} based on {\it pseudo quotient maps}. Pseudo quotient maps were first introduced in \cite{kobayashisaeki} and later defined again by \cite{kitazawa3}. 
\begin{Def}
\label{def:4}
Let $m>n$ be positive integers.
Let $f_q$ be a triangulable map from a closed manifold $M$ of dimension $m$ onto a polyhedron $W_q$ of dimension $n$ with respect
 to the corresponding canonical homeomorphisms. If for each point $p \in W_q$ and the interior of a small
 connected and closed neighborhood $N_p$ being $n$-dimensional and satisfying $p \in {\rm Int} N_p$, 
there exist a spherical fold map $f$, a point ${p}^{\prime} \in W_f$, a small connected and closed neighborhood $N_{{p}^{\prime}}$ with its interior containing the point, and
 a diffeomorphism $\Phi$ and a PL homeomorphism $\phi$ such that 
for the maps $f_{q,N_p}=f_q {\mid} {f_q}^{-1}(N_p):{f_q}^{-1}(N_p) \rightarrow N_p$ and $q_{f,N_{{p}^{\prime}}}:{q_f}^{-1}(N_{{p}^{\prime}}) \rightarrow N_{{p}^{\prime}}$, 
 the relation $q_{f,N_{{p}^{\prime}}} \circ \Phi =\phi \circ f_{q.N_p}$ holds, then $W_q$ is said to be a {\it pseudo quotient space} and $f_q$ is said to be a {\it pseudo normal spherical fold map}.
\end{Def}
We can naturally define a {\it singular} point, a {\it singular value}, a {\it regular value},  the {\it singular set}, the {\it singular value set}, the {\it regular value set} etc. of a pseudo spherical fold map.
There are some types of singular values of pseudo spherical fold maps and we introduce them in the following fundamental proposition. We can see the statements
 from fundamental explanations of articles on {\it Turaev`s shadow} such as \cite{martelli} and \cite{turaev} and ones on algebraic and differential topological properties of stable fold maps
 such as \cite{kobayashisaeki} \cite{saeki}, \cite{saeki2} and \cite{saekisuzuoka} for example: they are referenced later as important articles for fundamental and important facts
 related to such topics. 

\begin{Prop}
Let $f_p$ be a pseudo normal spherical fold map from a closed manifold $M$ of dimension $m>2$ onto a polyhedron $W_p$ of dimension $2$.
\begin{enumerate}
\item For a singular value $p \in W_p$, its inverse image has one or two singular points. In the latter case, we call
 the singular value a {\rm double point}. In the former case, we call the value a {\rm single point}.
\item For a single point $p \in W_p$, for a small regular neighborhood of $p$,
 the intersection of  the regular neighborhood and the regular value set consists of one or three connected components PL homeomorphic to the $2$-dimensional closed disc. 
 For the former case, we call $p$ a {\rm definite single point}.
\end{enumerate}
\end{Prop}

\section{Construction of special generic maps.}
We present a main theorem;  the definiton of a $2$-dimensional polyhedron realized as a target space of a pseudo normal spherical fold map {\it compatible with the natural polyhedron} will be explained later in Definition \ref{def:6}
\begin{Thm}
\label{thm:1}
For a pseudo normal spherical fold map $f_p$ on a $4$-dimensional
 closed and connected manifold $M$ onto a $2$-dimensional polyhedron $W_p$ {\rm compatible with the natural polyhedron}, there exists
 a $3$-dimensional compact and connected orientable smooth manifold $W_P$ satisfying the following.
\begin{enumerate}
\item $\partial W_P \neq \emptyset$. 
\item There exists a pseudo special generic map $f_P$ on $M$ onto $W_P$.
\item There exists a continuous map $g:W_P \rightarrow W_p$ satisfying $f_p=g\circ f_P$.
\end{enumerate} 
\end{Thm}
\begin{proof}
We construct $f_P$ and $g$ by constructing maps locally and glue them together so that the source manifold of $g$ is orientable through the following five steps. For some figures presented, see
 also the proof of Theorem 4.1 of \cite{saekisuzuoka}; same local models are depicted.
\\
\\
STEP 1 Around a double point.\\
As FIGURE \ref{fig:1} shows, we can construct a local map from a $4$-dimensional compact smooth manifold onto a $3$-dimensional compact and
 orientable manifold by piling two maps each of which is regarded as the product of a cobordism of special generic functions on $S^2$ and
 the identity map on $[-1,1]$. For cobordisms of special generic functions which will be used later again, see \cite{saeki5} for example and see \cite{milnor} for
 h-cobordisms including (h-)cobordisms of homotopy spheres. We also obtain the canonical
 projection from the $3$-dimensional manifold to the small regular neighborhood of the double point.
\begin{figure}
\includegraphics[width=35mm]{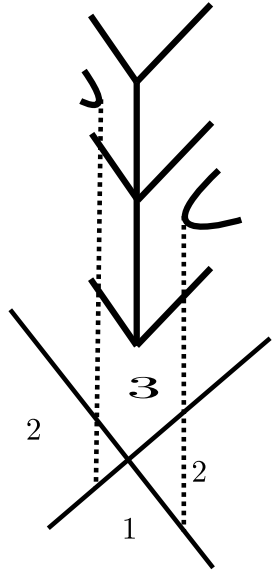}
\caption{Around a double point (with the number of connected components of inverse images).}
\label{fig:1}
\end{figure} 

A {\it Y-bundle} is a bundle whose fiber is a Y-shaped $1$-dimensional polyhedron and we introduce terminologies on monodromies of Y-bundles appearing as subsets of the target spaces of pseudo normal spherical fold maps.

\begin{Def}
\label{def:6}
For an oriented simple loop on the singular value set of a pseudo normal spherical fold map, we can
 canonically obtain a Y-bundle over the loop respecting local forms presented in FIGURE \ref{fig:2} and FIGURE \ref{fig:3}.  
If the structure group or the monodromy is trival, then the monodromy of the Y-bundle is said to be {\it trivial} and if the
 PL homeomorphism corresponding to the monodromy does not fix points except the point in the center, then the monodromy is said to be {\it free}.

Last, if for any oriented simple loop on the singular value set of a pseudo normal spherical fold map, the monodromy of the canonically corresponding Y-bundle is trivial or free, then the target space is said to be {\it compatible with the natural orientation}.
\end{Def}

\begin{figure}
\includegraphics[width=35mm]{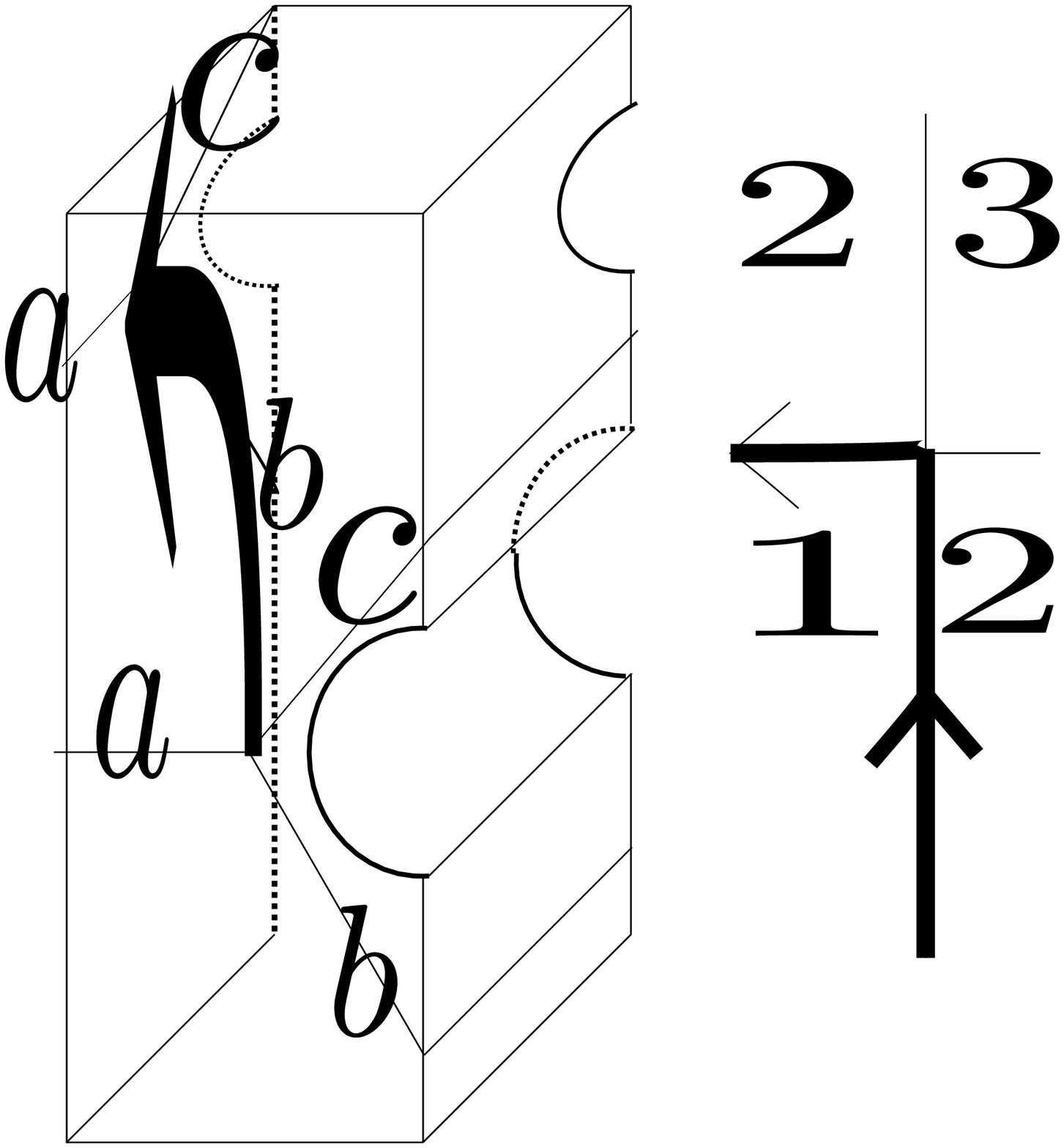}
\caption{A local form of a Y-bundle near a double point (each integer represents the number of connected components of each inverse image).}
\label{fig:2}
\end{figure}
\begin{figure}
\includegraphics[width=35mm]{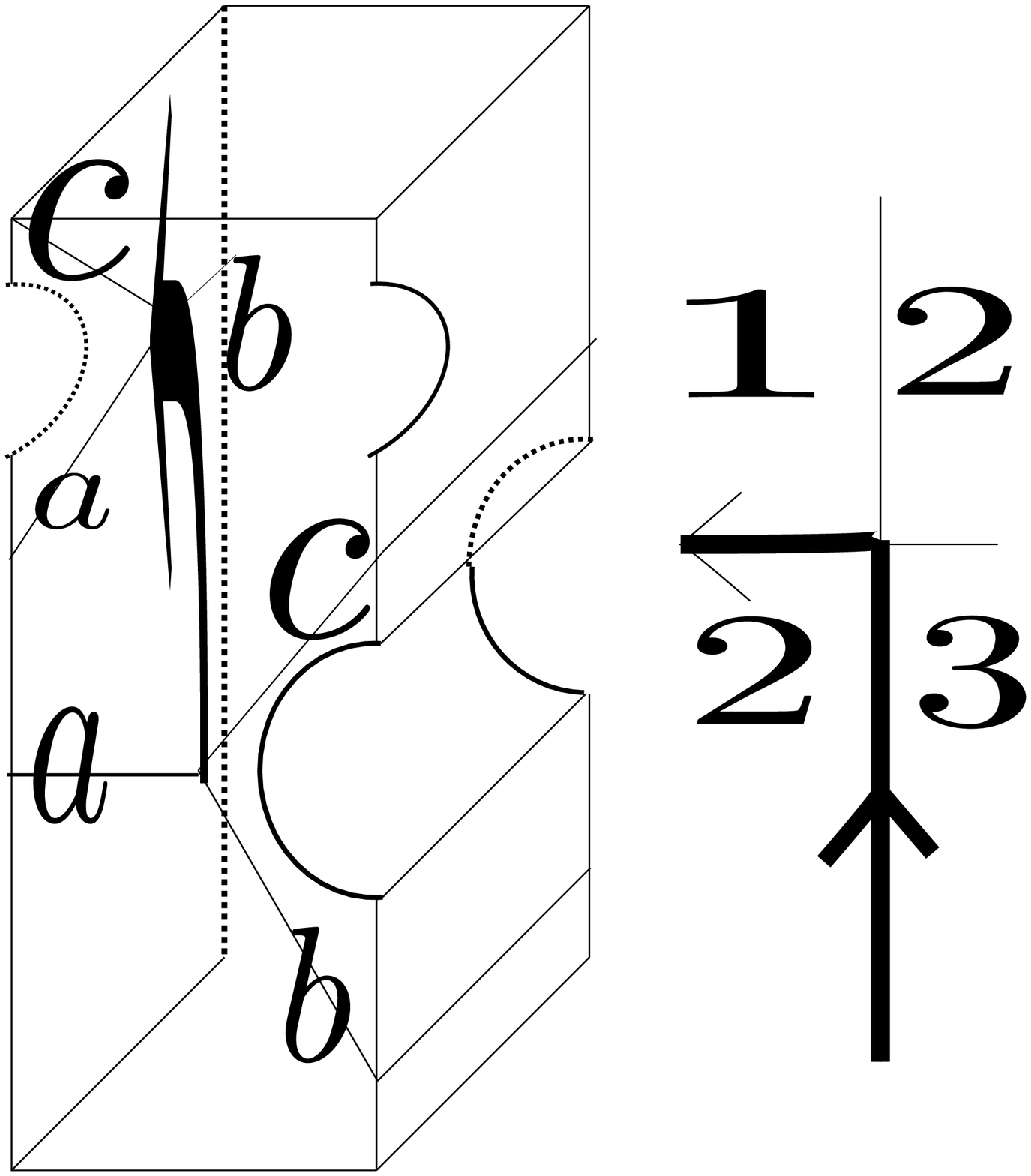}
\caption{A local form of a Y-bundle near a double point (each integer represents the number of connected components of each inverse image).}
\label{fig:3}
\end{figure}

Note that such terminologies are defined for {\it simple polyhedra}, essentially same as target spaces of pseudo normal spherical fold maps and defined in Definition \ref{def:7}.\\
\\
STEP 2 Around a single point.\\
As FIGURE \ref{fig:4} or \ref{fig:5} shows, we can construct a local map from a $4$-dimensional compact smooth manifold onto a $3$-dimensional compact and
 orientable manifold by the product of a cobordism of special generic functions on $S^2$ and
 the identity map on $[-1,1]$.
By the assumption that a small regular neighborhood of the singular value set of the pseudo normal spherical fold map contains no non-orientable surface and
 the cobordism of special generic maps are invariant under a diffeomorphism on
 the whole space smoothly isotopic to a given diffeomorphism sending each connected component of the boundary onto the original component or a component different from
 the, original one we can attach obtained maps and manifolds together with the ones obtained in the previous step compatibly.
\begin{figure}
\includegraphics[width=35mm]{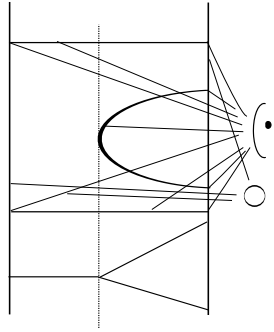}
\caption{Around a single point (points and circles are inverse images of the cobordism and the monodromy of the Y-bundle  considered here is trivial).}
\label{fig:4}
\end{figure} 
\begin{figure}
\includegraphics[width=35mm]{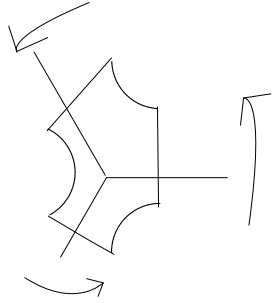}
\caption{Around a single point (the monodromy of the Y-bundle considered here is free).}
\label{fig:5}
\end{figure} 

\noindent STEP 3 Around a definite single point. \\
As FIGURE \ref{fig:6} shows, we can construct a local map from a $4$-dimensional compact smooth manifold onto a $3$-dimensional compact and
 orientable manifold by the product of a cobordism between a special generic function on $S^2$ and a function on the empty set and
 the identity map on $[-1,1]$. By a discussion similar to that of the previous step, we can attach the obtained objects to the constructed ones compatibly. 
\begin{figure}
\includegraphics[width=35mm]{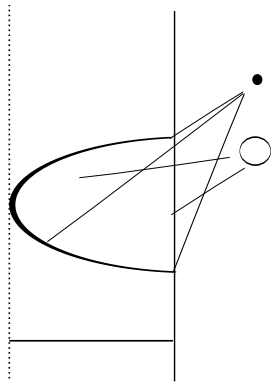}
\caption{Around a normal single point (points and circles are inverse images of cobordism of special generic functions).}
\label{fig:6}
\end{figure} 
\\
\\
We have constructed maps on the singular value set and its small neighborhood.
We construct maps on the remained part of the regular value set by extending the constructed map.\\
\\
STEP 4 Around a $1$-dimensional skelton on the regular value set. \\
We extend the existing map to each $1$-dimensional skelton. To make the resulting $3$-dimensional compact manifold orientable, we need to construct the local map
 as FIGURE \ref{fig:7} shows (see \cite{kitazawa5} for a precise explanation on this part); more precisely, we need to be careful in the case where the regular value set
 contains a non-orientable connected component. 
\begin{figure}
\includegraphics[width=35mm]{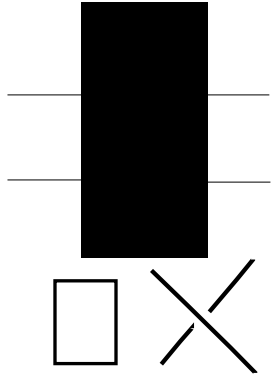}
\caption{Construction around a $1$-dimensional skeleton on the regular value set. Note that we may
 need to twist the target closed interval of the function to make the obtained $3$-dimensional manifold orientable and the lower two figures represent the local form of the black part.}
\label{fig:7}
\end{figure} 
\\
\\
STEP 5 Around a $2$-dimensional skelton on the regular value set. \\
On each $2$-dimensional cell $D$, we can construct a product of a special generic function on $S^2$ and the identity map on the cell $D$. Consider the attaching
 map on the boundary $\partial D \times S^2$. It is a bundle isomorphism on the trivial $S^2$-bundle over $\partial D^2$. Moreover, it is regarded as a bundle
 isomorphism preserving the Morse functions since the natural homomorphism ${\pi}_1(SO(2)) \rightarrow {\pi}_1(SO(3))$ is known to be surjective and
 the diffeomorphism groups of $S^1$ and $S^2$ are regarded as linear or homotopy equivalent to $O(2)$ and $O(3)$, respectively. The discussion enables us to extend
 the constructed maps and obtain a $3$-dimensional compact and connected orientable manifold. See also FIGURE \ref{fig:8}. \\
\begin{figure}
\includegraphics[width=35mm]{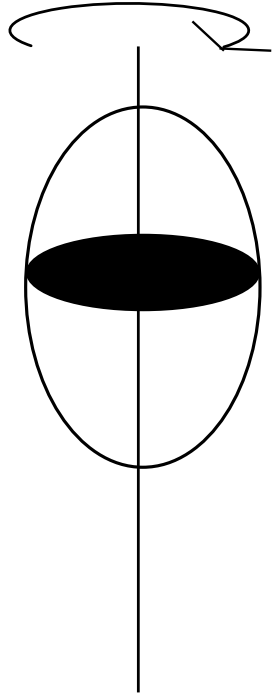}
\caption{Construction around a $2$-dimensional skeleton on the regular value set.}
\label{fig:8}
\end{figure} 

We can obtain the resulting maps and manifolds.
\end{proof}
\begin{Rem}
\label{rem:1}
We note about Shibata's work \cite{shibata}, in which essentially same figures as some figures in the proof of Theorem \ref{thm:1} appear.
Shibata investigated the condition for a smooth map on a $3$-dimensional compact and orientable smooth manifold with non-empty boundary into the plane to be
 represented as the composition of an immersion into ${\mathbb{R}}^3$ and the canonical projection onto the plane. In the proof, as a partially similar but essentially
 different work, we construct a $3$-dimensional compact and orientable smooth manifold and a continuous map onto a given $2$-dimensional polyhedron realized as the target space
 of a pseudo normal spherical fold map compatible with the natural orientation. For example, note that the $2$-dimensional polyhedron may not be realized as the Reeb space of a spherical fold map into the plane; consider the case where the regular value set contains a non-orientable connected component.
\end{Rem}
\begin{Rem}
\label{rem:2}
In the present paper, as fold maps, stable fold maps are considered. Generally, a {\it stable} map is defined
 as a map such that by slight perturbations, the resulting maps are always $C^{\infty}$ {\it equivalent} to the original map; as another
 definition, it is defined as a smooth map such that there exists an open neighborhood of the map and the maps there are always $C^{\infty}$ equivalent to the original one
 where the topology of the spaces consisting of all smooth maps is the $C^{\infty}$ Whitney topology. 

Two smooth maps $f_1$ and $f_2$ are said to be $C^{\infty}$ {\it equivalent} if there exists a pair $(\Phi,\phi)$ of diffeomorphisms of the source manifolds and of the target manifolds and the relation $\phi \circ f_1=f_2 \circ \Phi$ holds. 

If a stable map from a closed manifold of dimension larger than $1$ into the plane is stable, then each singular point is of the same form as a that of a fold map or a {\it cusp}. 
The set of all singular points whose forms are same as that of a fold map and that of a singular point of a fixed index is a smooth submanifold of codimension $1$ and the
 number of cusps is finite. Moreover, if the restriction of the map to the set of all the singular points whose form is same as that of a fold map is transversal and at distinct
 cusps, the values are different and the inverse image of the value of each cusp has only one singular point, then the map is stable and stable maps into the plane on manifolds whose dimensions are larger than $1$ always have such differential topological properties. 

For these elemental terminologies and facts,  see \cite{golubitskyguillemin} for example. For stable maps into the plane, see also \cite{thom} and \cite{levine} for example.

Moreover, we can define the extension of pseudo normal spherical fold map considering not only spherical fold maps but also stable maps such that inverse images satisfy
 same conditions and we can prove Theorem \ref{thm:1}: we must consider about cusps and for this, see FIGURE \ref{fig:9} and also Figure 6 and Figure 7 of \cite{saekisuzuoka} 

\begin{figure}
\includegraphics[width=35mm]{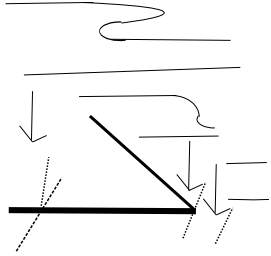}
\caption{Around a point corresponding to a cusp (the inverse image of a point in the three depicted surfaces in the upper part is a point or a circle).}
\label{fig:9}
\end{figure} 
\end{Rem}
\begin{Cor}
\label{cor:1}
A $4$-dimensional closed and connected manifold $M$ admitting a pseudo normal spherical fold map onto a $2$-dimensional polyhedron compatible with the natural orientation admits a special generic map into ${\mathbb{R}}^3$. Moreover, by composing a canonical projection onto the plane, we have a normal spherical fold map into the plane and the resulting Reeb space is compatible with the natural orientation.
\end{Cor}
\begin{proof}
As Theorem \ref{thm:1}, we can obtain a pseudo special generic map onto a $3$-dimensional orientable manifold with non-empty boundary. By a well known fact, we can immerse
 the target manifold into ${\mathbb{R}}^3$ and we have a special generic map. The latter follows from fundamental theory of \cite{saekisuzuoka} and \cite{shibata}. To remove
 cusps mentioned in Remark \ref{rem:2}, apply the theory of \cite{levine} and the fact that the Euler number of the source manifold is even.
\end{proof}

\begin{Prop}[\cite{kitazawa}, \cite{kitazawa2} and \cite{saeki} etc.]
\label{prop:2}
A pseudo normal spherical fold map $f_p$ from a $4$-dmensional closed and connected manifold $M$ onto a $2$-dimensional polyhedron $W_p$ induces isomorphisms
 between the 1st homology and homotopy groups. Moreover, the group $H_2(W_f;\mathbb{Z})$ is free and the rank
 of the group $H_2(M;\mathbb{Z})$ is twice the rank of that of $H_2(W_f;\mathbb{Z})$.
If the dimension of the source manifold is $m>3$, then a pseudo normal spherical fold map $f_p$ from an $m$-dimensional closed and connected manifold $M$ onto a $2$-dimensional polyhedron $W_p$ induces isomorphisms
 between the $k$-th homology and homotopy groups for $0 \leq k \leq m-3$.
\end{Prop}
\begin{Rem}
\label{rem:3}
The proposition above is proven as a proposition for spherical fold maps or more generally, {\it spherical} stable maps explained also in Remark \ref{rem:2} from $4$-dimensional closed and connected manifolds into $2$-dimensional manifolds with no boundaries in \cite{saekisuzuoka}. However, by the manner of
 the proof, we have the mentioned result.
\end{Rem}
This means that in the situation of Theorem \ref{thm:1} and Corollary \ref{cor:1}, for the target Reeb spaces or
 the target pseudo quotient spaces for given manifolds, 1st homology and homotopy groups and the ranks of 2nd homology groups are invariant.

%\begin{Thm}
%\end{Thm}
In Example \ref{ex:3} and Example \ref{ex:4}, we introduce simple examples.
The numbers in each connected component of the regular value set in the figures represents the numbers of the inverse images of the connected components.
\begin{Ex}
\label{ex:3}
FIGURE \ref{fig:10} shows two kinds of fold maps. They are round fold maps, defined as fold maps whose singular value sets are embedded concentric spheres and introduced and systematically studied in \cite{kitazawa} and \cite{kitazawa2} for example. The source manifolds are $S^4$ and the total space of an $S^2$-bundle over $S^2$, respectively. Note that both trivial and non-trivial bundles admit fold maps like ones the right figure shows. We can apply Theorem \ref{thm:1} to the maps.
FIGURE \ref{fig:11} show a fold map which is not round. The source manifold is represented as a connected sum of two total spaces of $S^2$-bundles over $S^2$. We can apply Theorem \ref{thm:1} to this case.
\begin{figure}
\includegraphics[width=35mm]{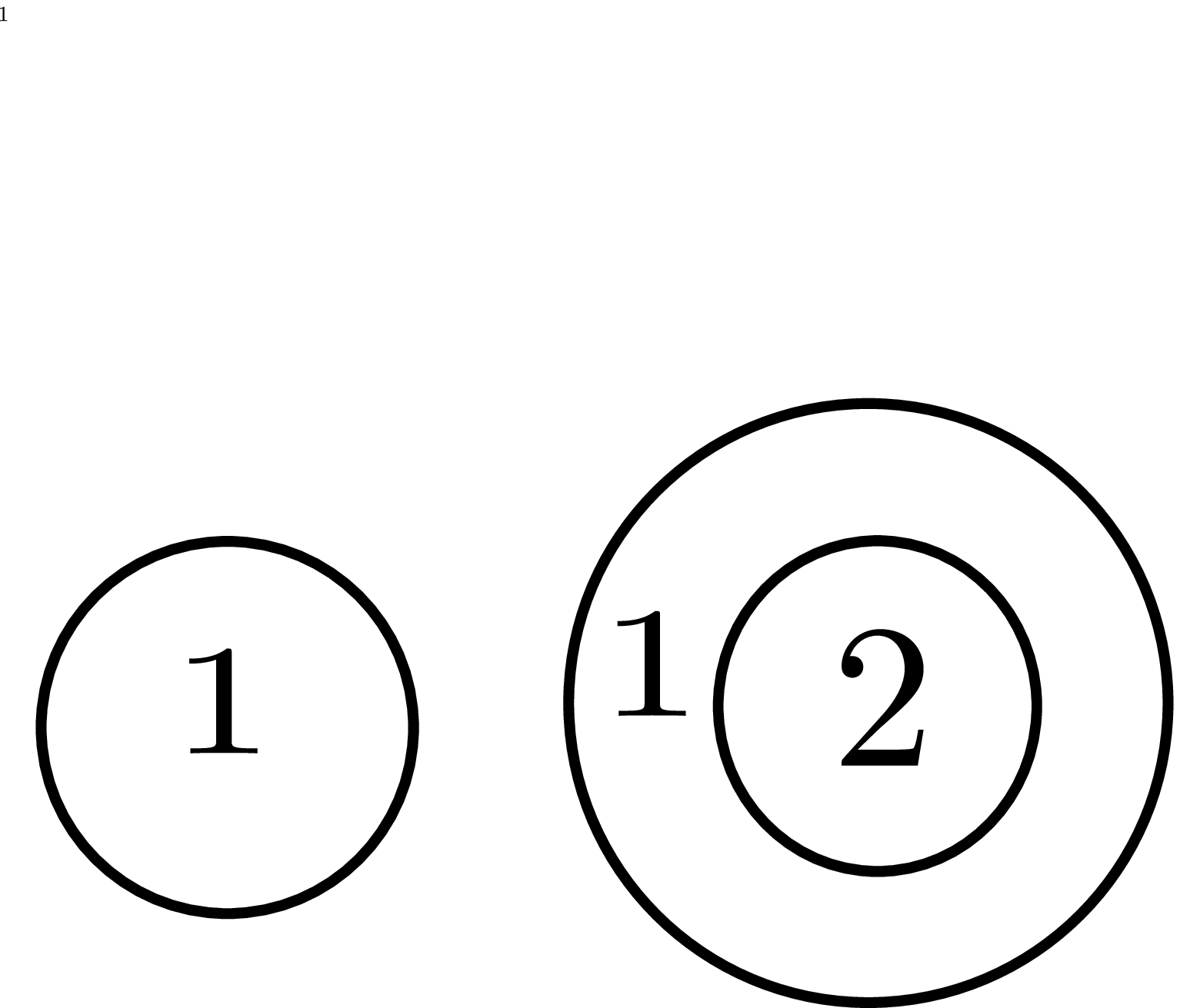}
\caption{Two round fold maps.}
\label{fig:10}
\end{figure} 
\begin{figure}
\includegraphics[width=35mm]{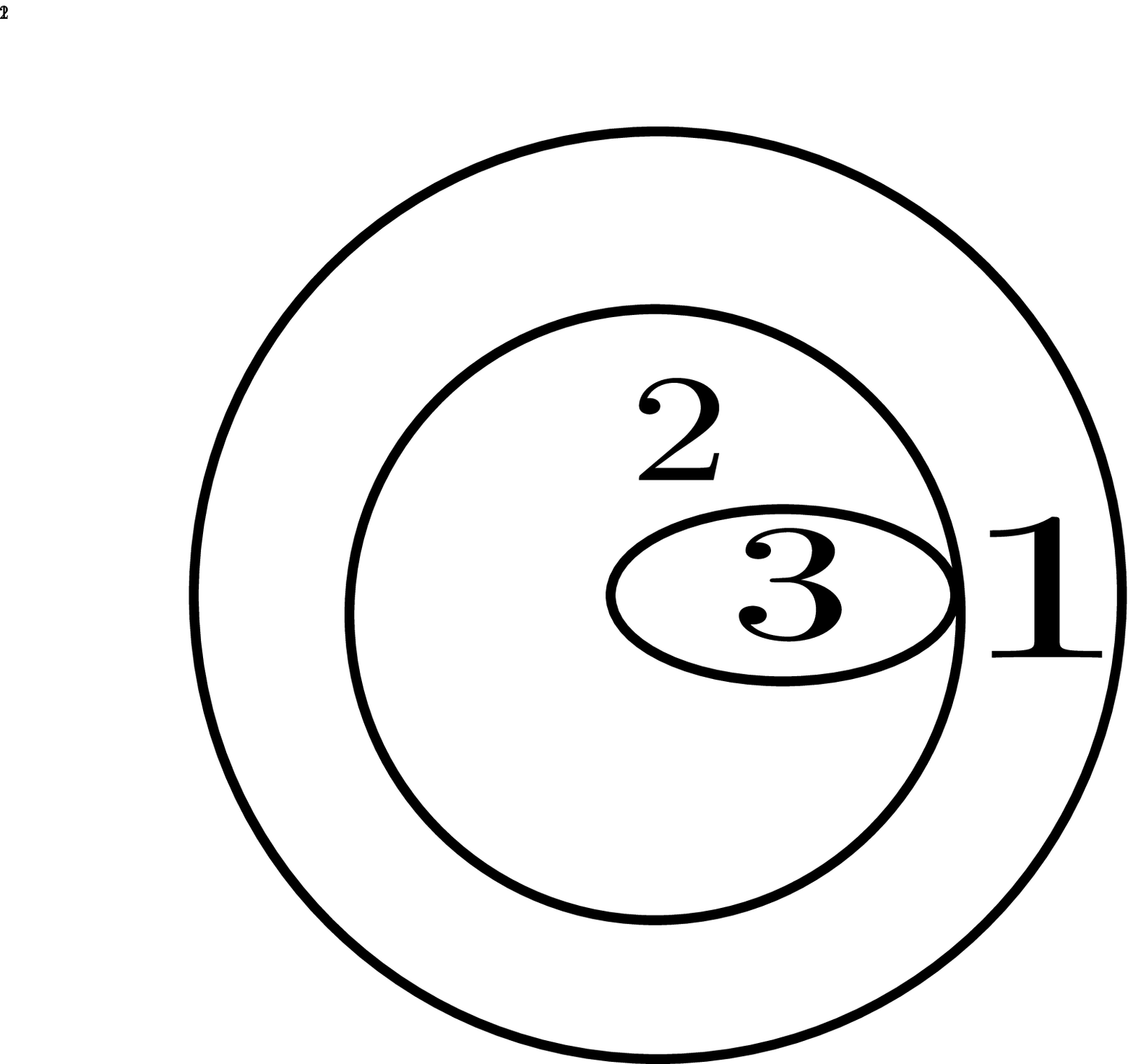}
\caption{Another example.}
\label{fig:11}
\end{figure} 
\end{Ex}
\begin{Ex}[\cite{suzuoka}]
\label{ex:4}
An example by Suzuoka dipicted as FIGURE \ref{fig:12} is studied in \cite{suzuoka}. The source manifold is $S^2 \times S^2$ or the total
 space of a non-trivial $S^2$-bundle over $S^2$. We also note that these manifolds admit fold maps as dipicted. We can apply Theorem \ref{thm:1} to this case.
\begin{figure}
\includegraphics[width=60mm]{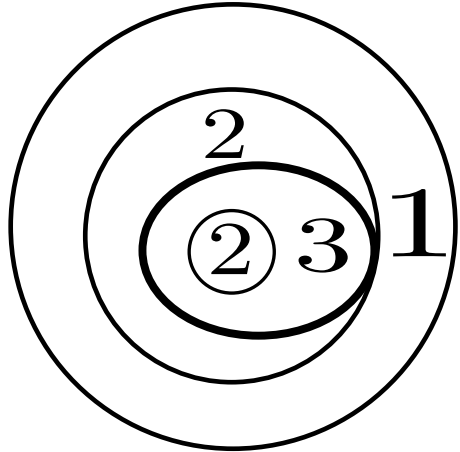}
\caption{Suzuoka`s example of spherical fold maps}
\label{fig:12}
\end{figure} 
\end{Ex}

\section{Differential topological meanings of the result and related facts and problems.}
\subsection{Manifolds admitting special generic maps}
As mentioned in the abstract and the introduction or section 1, manifolds admitting special generic maps into Euclidean spaces are often restricted if the dimesions of the target spaces
 are not so low.
For example, it is known that if a homotopy sphere of dimension $m>3$ admits a special generic map into ${\mathbb{R}}^n$ satisfying $n=m-1,m-2,m-3$, then it is the standard sphere.
For the $4$-dimensional case, first, a $4$-dimensional closed and connected manifold admits a special generic map into the plane if and only if the manifold is represented as a connected sum of the total
 spaces of $S^3$-bundles over $S^1$ ($S^4$ is also included in the trivial case). Moreover, a $4$-dimensional closed and connected manifold whose fundamental group is represented as a free
 product of a finite number of $\mathbb{Z}$ admits a special generic map into ${\mathbb{R}}^3$ if and only if it is represented as a connected sum of the total spaces of $S^3$-bundles over $S^1$ and $S^2$-bundles over $S^2$ ($S^4$ is included). 
According to an informal comment by Saeki in 2010, the following conjecture on the rigidity of diffeomorphism types of $4$-dimensional manifolds admitting special generic maps seems to be true.
\begin{Conj}
If two $4$-dimensional homeomorphic closed and connected manifolds admit special generic maps into an Euclidean space of dimension smaller than $4$, then they are diffeomorphic.
\end{Conj} 
Theorem \ref{thm:1} implies that the class of source manifolds of maps of the considered class, which seems to be wider than the class of $4$-dimensional closed and connected
 manifolds admitting special generic maps, may not be so large. In fact, Saeki and Suzuoka \cite{saekisuzuoka} presented a problem on the diffeomorphism types of $4$-dimensional closed manifolds
 admitting spherical fold maps into the plane questioning whether a manifold homeomorphic and not diffeomorphic to a manifold represented as connected sums of the
 total spaces $S^2$-bundles
 over $S^2$ such as the Moishezon-Teicher surface with zero signature admits a spherical fold map into the plane: according to Theorem \ref{thm:1}, it is false under the assumption that the resulting quotient map onto the Reeb space is standard. 
See also \cite{saekisakuma} and \cite{saekisakuma2} for such topics on diffeomorphism types.

Note that this is false for pairs of general dimensions. Every homotopy sphere of dimension larger than $1$ except $4$-dimensional exotic homotopy spheres
 admits a special generic map into the plane (\cite{saeki2}). 

Last, as a related result, recently, Wrazidlo \cite{wrazidlo} has shown that $7$-dimensional oriented homotopy spheres whose oriented diffeomorphism types
 belong to a family of $14$ types of 28 types do not admit special generic maps into ${\mathbb{R}}^3$.

\subsection{$2$-dimensional pseudo quotient spaces called simple polyhedra including the
 cases of Turaev`s shadows.}
\begin{Def}
\label{def:7}
A {\it simple}
polyhedron is a $2$-dimensional polyhedron locally PL homeomorphic to a closed neighborhood of a point in the target space of a pseudo normal spherical fold map on $4$-dimensional closed manifold onto the $2$-dimensional polyhedron. 
\end{Def}

\begin{Rem}
\label{rem:4}
We can replace the dimension $4$ by any other integer larger than $4$. We cannot replace it by $2$ or $3$ : for the $3$-dimensional case, some {\it shadows}, explained more precisely later, are regarded as pseudo normal spherical fold map. We can know this from fundamental explanations of \cite{saeki3} and \cite{saeki4} and also from \cite{costantinothurston} and \cite{ishikawakoda} for example. 
\end{Rem}

We state the following fundamental fact as a proposition.
\begin{Prop}
Let $W_p$ be a simple polyhedron.
For any integer $m>2$, there exists a closed smooth manifold of dimension $m$ and a pseudo normal spherical fold map onto the polyhedron $W_p$.
\end{Prop}

As presented in Remark \ref{rem:4}, some {\it shadows}, regarded as pseudo quotient maps, have been introduced by Turaev \cite{turaev} and they explain for the case $m=3$ with the source manifolds orientable. 
Costantino and Thurston have made use of the fact that Reeb spaces of stable fold maps from $3$-dimensional
 closed orientable manifolds into surfaces without connected components of inverse images called II$^3$ type, are
 regarded as shadows in \cite{costantinothurston}. Recently a related study \cite{ishikawakoda} has been done. 
There are various related studies on shadows and manifolds admitting shadows 
whose singular value sets are disjoint unions of circles or whose singular value sets have intersections fixing the number of intersections. For example, Martelli studied manifolds (more precisely, $4$-dimensional closed smooth manifolds obtained in canonical ways) admitting shadows systematically in \cite{martelli} and later in \cite{kodamartellinaoe} with Koda and Naoe. 
In \cite{naoe}, \cite{naoe2} and \cite{naoe3}, Naoe systematically and explicitly studied contractible shadows and $3$-dimensional
 manifolds admitting such shadows.
In these studies, there appear various explicit shadows for which we cannot apply Theorem \ref{thm:1}.

As an explicit fact, the following is also known. It comes from a study of Kobayashi and Saeki \cite{kobayashisaeki} and that of Naoe \cite{naoe}, which are mutually independent studies, and
 some additional discussions.
\begin{Fact}[\cite{kobayashisaeki} and \cite{naoe} etc.]
\label{fact:3}
Let $M$ be a closed and connected manifold of dimension $m>2$ admitting a pseudo spherical fold map onto a $2$-dimensional simple polyhedron $W_p$ such that the singular value
 set is a disjoint union of circles or contains no double points. If $m\geq 4$ holds, $M$ is
 simply-connected and $H_2(W_p;\mathbb{Z})$ is zero, then $W_p$ is contractible and can be collapsed to the $2$-dimensional closed disc and $M$ is a homotopy sphere. Especially, for $m=4,5,6$, $M$ is the standard sphere.
If $m=3$ holds, $W_p$ is simply-connected and $H_2(W_p;\mathbb{Z})$ is zero, then $W_p$ is contractible and can be collapsed to the $2$-dimensional closed disc and $M$ is diffeomorphic to $S^3$.
\end{Fact}
We remark on \cite{naoe}.
\begin{Rem}
Naoe \cite{naoe} has shown that a $2$-dimensional simple polyhedron $W_p$ without double points being simply-connected and satisfying $H_2(W_p;\mathbb{Z}) \cong \{0\}$ can be collapsed to the $2$-dimensional closed disc based on the study \cite{martelli}, in which representations of shadows or simple polyhedra with the singular value sets being disjoint unions of circles or having no double points by labeled graphs have been introduced and presented. Note also that a main purpose of \cite{naoe} is to show that a $4$-dimensional compact manifold bounded by the source $3$-dimensional manifold obtained by a canonical
 procedure based on the fundamental theory of shadows is diffeomorphic to $D^4$ and that \cite{naoe2} is also on such studies of $4$-dimensional manifolds including $3$-dimensional manifolds appearing
 as the boundaries. 
\end{Rem}
\begin{figure}
\includegraphics[width=35mm]{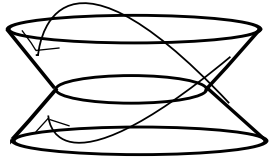}
\caption{An example satisfying the assumption of Fact \ref{fact:3} such that the target $2$-dimensional polyhedron is not compatible with the natural orientation and for which we cannot apply Theorem \ref{thm:1} (see also section 7 of \cite{kobayashisaeki}).}
\label{fig:13}
\end{figure}
Note that in the case of Fact \ref{fact:3}, the singular value set is a disjoint union of spheres and that we cannot always apply Theorem \ref{thm:1}; for example, to the case of FIGURE \ref{fig:13}, appearing in \cite{kobayashisaeki}. 

We have the following.

\begin{Thm}
\label{thm:2}
The source manifold of a pseudo normal spherical fold map onto a $2$-dimensional simple polyhedron which is simply-connected and whose second homology group with coefficient ring $\mathbb{Z}$ vanishes is a homotopy sphere if the dimension of the source manifold is larger than $3$ and in the case where the dimension of the source manifold is $5$ or $6$, the source manifold is diffeomorphic to a standard sphere. In addition, if the simple polyhedron is compatible with the natural orientation, then also for the case where the source manifold is $4$-dimensional, the source manifold is a standard sphere or $S^4$.
\end{Thm}
\begin{proof}
We can know this from Proposition \ref{prop:2} except the latter part for the $4$-dimensional case or the fact that the source manifold is $S^4$ in the latter situation. We can
 show this also for the $4$-dimensional case by seeing that we can apply Theorem \ref{thm:1} to this case and that
 the $4$-dimensional homotopy sphere admitting a special generic map into ${\mathbb{R}}^3$ is diffeomorphic to $S^4$ . 
\end{proof}
\begin{figure}
FIGURE \ref{fig:14} shows an explicit case.
\includegraphics[width=35mm]{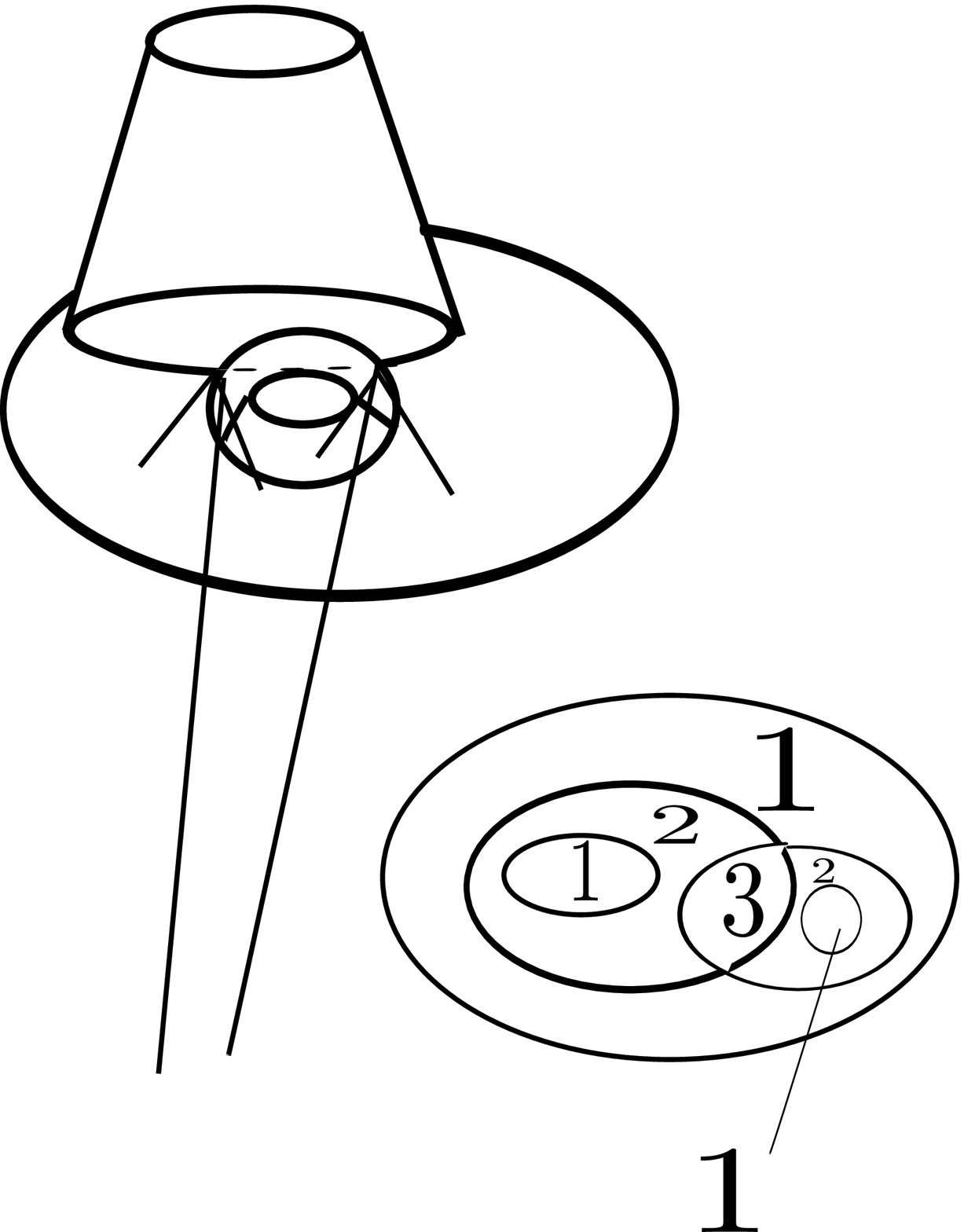}
\caption{A non-trivial example: arrows indicate two crossings in the singular value set, consisting of spheres and the right figure shows a realization by a fold map into
 the plane with the number representing the number of the connected components of the inverse images.}
\label{fig:14}
\end{figure}

\begin{Ex}
\label{ex:4}
Bing's house is a contractible simple polyhedron, compatible with the natural orientation. It is known that the singular value set has $2$ double points.
\end{Ex}
For the case where the source manifold is $3$-dimensional in the theorem,  see \cite{naoe2} for example; the source manifold may not be a homotopy sphere.
\begin{Rem}
In \cite{suzuoka}, Suzuoka determined the diffeomorphism types of manifolds admitting spherical fold maps into the plane by considering the open book structures or bundle over the circle with fibers being inverse images of Morse functions introduced from the structures of the fold maps and applying Kirby diagrams of $4$-dimensional manifolds (FIGURE \ref{fig:15}). He investigated simplest generalizations of spherical fold maps regarded as round fold maps.
We can apply this technique in the cases FIGURE \ref{fig:10}, FIGURE \ref{fig:11} and FIGURE \ref{fig:12}.
However, we cannot apply the method for the case of FIGURE \ref{fig:14}.
\end{Rem}
\begin{figure}
\includegraphics[width=35mm]{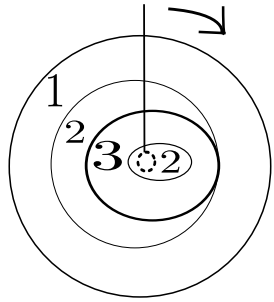}
\caption{Suzuoka`s ideas for determining source manifolds (Morse functions parametrized by points in the dotted circle give a bundle over the dotted circle whose fibers are inverse images of the Morse functions).}
\label{fig:15}
\end{figure}
\section{Results and Remarks for higher dimensions.}
\begin{Thm}
\label{thm:3}
Corollary \ref{cor:1} holds also in the case where the source manifold is $5$-dimensional.
\end{Thm}
\begin{proof}[Sketch of the proof.]
Even though we must care about discussions on monodromies along loops, by virtue of the fact that differentiable strutures of $5$-dimensional manifolds are unique, by changing maps into suitable ones
 without changing source manifolds, we can construct maps and manifolds similarly through STEPS 1--4. In STEP 5, note that the diffemorphism group of $S^3$ is homotopy equivalent to $O(4)$ and the
 corresponding homomorphism ${\pi}_1{(SO(2))} \rightarrow {\pi}_1(SO(4))$ (consider a special generic Morse function on $S^3$) between the corresponding homotopy groups are surjective also in this case.
\end{proof}
\begin{Ex}
In \cite{saeki2}, Saeki has shown that a $5$-dimensioal closed and simply connected
 manifold, completely classfied in \cite{barden}, admits a special generic map into ${\mathbb{R}}^3$ if and only if it is represented
 as a connected sum of total $S^3$-bundles over $S^2$. This condition is same as that for a such a manifold to admit a special generic map into ${\mathbb{R}}^4$ proven
 in \cite{nishioka}. Note that this condition on the manifold is equivalent to the assumption that the 2nd homology group is free. 
\end{Ex}
We cannot extend Theorem \ref{thm:1}, Corollary \ref{cor:1} and \ref{thm:3} generally. As presented, according to Saeki, $7$-dimensional homotopy spheres admitting a special generic map into ${\mathbb{R}}^4$ must be standard and Wrazidlo \cite{wrazidlo} has shown that there are at least $14$ oriented diffeomorphism types of $7$-dimensional homotopy spheres such that the homotopy spheres do not admit special generic maps into ${\mathbb{R}}^3$. However, in Fact \ref{fact:3}, the source manifold can be any homotopy sphere except exotic $4$-dimensional spheres; as a simplest example, consider a special generic map whose singular set is an
 embedded circle into the plane
 constructed in \cite{saeki2}.

\end{document}